\documentclass{amsart}
\usepackage[T1]{fontenc}
\usepackage[utf8]{inputenc}
\usepackage{amsfonts}
\usepackage{graphicx}

\usepackage{geometry}
\usepackage{amssymb,amsmath,amsfonts,amsxtra,amsthm}
\usepackage{xcolor}
\usepackage[all]{xy}
\usepackage{multicol}
\usepackage{multirow}
\usepackage{array}
\usepackage{booktabs}
\usepackage{lscape}
\usepackage[mathscr]{euscript}
\usepackage{graphicx}
\usepackage[T1]{fontenc}
\usepackage[utf8]{inputenc}
\usepackage{amsfonts}
\usepackage{graphicx}

\usepackage{geometry}
\usepackage{amssymb%,amsmath
,amsfonts,amsxtra
}

\usepackage{xcolor}
\usepackage[all]{xy}
\usepackage{multicol}
\usepackage{multirow}
\usepackage{array}
\usepackage{booktabs}
\usepackage{lscape}
\usepackage[mathscr]{euscript}
\usepackage{enumerate}

\newtheorem{theorem}{Theorem}[section]
\newtheorem{lemma}[theorem]{Lemma}

\newtheorem{corollary}[theorem]{Corollary}

\theoremstyle{definition}
\newtheorem{definition}[theorem]{Definition}
\newtheorem{example}[theorem]{Example}

\theoremstyle{remark}
\newtheorem{remark}[theorem]{Remark}

\numberwithin{equation}{section}

\newbox\pullbackbox
\setbox\pullbackbox=\hbox{\xy 0;<1mm,0mm>: \POS(4,0)\ar@{-} (0,0) \ar@{-} (4,4)
\endxy}
\def\pullback{\copy\pullbackbox}

\makeatletter
\let\c@proposition\c@theorem
\let\c@corollary\c@theorem
\let\c@lemma\c@theorem
\let\c@definition\c@theorem
\let\c@example\c@theorem
\let\c@remark\c@theorem
\makeatother

\def\ot{\otimes}

\begin{document}

\title{A torsion theory in the category of cocommutative Hopf algebras}

\author{Marino Gran}
\author{Gabriel Kadjo}
\address[Marino Gran, Gabriel Kadjo]{Institut de Recherche en Mathématique et Physique, Université catholique de Louvain, Chemin du Cyclotron 2, 1348 Louvain-la-Neuve, Belgium}
\email{marino.gran@uclouvain.be}          
\email{gabriel.kadjo@uclouvain.be}       

\author{Joost Vercruysse}
\address[Joost Vercruysse]{D\'epartement de Math\'ematique, Universit\'e Libre de Bruxelles, Boulevard du Triomphe, 1050 Bruxelles, Belgium.}
\email{jvercruy@ulb.ac.be}



\date{\today}

\maketitle

\begin{abstract}
The purpose of this article is to prove that the category of cocommutative Hopf $K$-algebras, over a field $K$ of characteristic zero, is a semi-abelian category. Moreover, we show that this category is action representable, and that it contains a torsion theory whose torsion-free and torsion parts are given by the category of groups and by the category of Lie $K$-algebras, respectively.
\end{abstract}

\section{Introduction}
\label{intro}
The starting point of this article on Hopf algebras is a well-known result due to A. Grothendieck, as outlined in \cite{Sweedler}, saying that the category of finite-dimensional, commutative and cocommutative Hopf $K$-algebras over a field $K$ is abelian. This result was extended by M. Takeuchi to the category of commutative and cocommutative Hopf $K$-algebras, not necessarily finite-dimensional \cite{Takeuchi}. 
The category $\mathbf{Hopf}_{K,coc}$ of cocommutative Hopf $K$-algebras is not additive, thus it can not be abelian. In the present article we investigate some of its fundamental exactness properties, showing that it is a homological category (Section \ref{homological}), and that it is Barr-exact (Section \ref{exact}), leading to the conclusion that the category $\mathbf{Hopf}_{K,coc}$ is semi-abelian \cite{Janelidze2002367} when the field $K$ is of characteristic zero (Theorem \ref{semi-abelian}). This result establishes a new link between the theory of Hopf algebras and the more recent one of semi-abelian categories, both of which can be viewed as wide generalizations of group theory.
Since a category $\mathbf{C}$ is abelian if and only if both $\mathbf{C}$ and its dual $\mathbf{C}^{op}$ are semi-abelian, this observation can be seen as a ``non-commutative'' version of Takeuchi's theorem mentioned above. The fact that the category $\mathbf{Hopf}_{K,coc}$ is semi-abelian was independently obtained by Clemens Berger and Stephen Lack. A recent article of Christine Vespa and Marc Wambst shows that the abelian core of $\mathbf{Hopf}_{K,coc}$ is the category of commutative and cocommutative Hopf $K$-algebras \cite{Vespa}.

 In the present work we also prove the existence of a non-abelian torsion theory $(\mathbf{T},\mathbf{F})$ in $\mathbf{Hopf}_{K,coc}$ (Theorem \ref{hereditary torsion}), where the torsion subcategory $\mathbf{T}$ is the category of \emph{primitive Hopf K-algebras}, which is isomorphic to the category of Lie $K$-algebras, and the torsion-free subcategory 
$\mathbf{F}$ is the category of \emph{group Hopf K-algebras}, which is isomorphic to the category of groups.

The categories of groups and of Lie $K$-algebras are two typical examples of semi-abelian categories: this shows again that the theories of cocommutative Hopf algebras and of semi-abelian categories are strongly intertwined.
The category $\mathbf{Hopf}_{K,coc}$ inherits 
some fundamental exactness properties from groups and Lie algebras thanks to the well-known canonical decomposition of a cocommutative Hopf algebra into a semi-direct product of a group Hopf algebra and a primitive Hopf algebra (a result associated with the names Cartier-Gabriel-Kostant-Milnor-Moore).
The present work opens the way to some new applications of categorical Galois theory \cite{Janelidze} in the category of cocommutative Hopf $K$-algebras, since the reflection from this category to the torsion-free subcategory of group Hopf algebras enjoys all the properties needed for this kind of investigations, as we briefly explain in Section \ref{torsion}. We conclude the article by observing that the semi-abelian category $\mathbf{Hopf}_{K,coc}$ is also an action representable category \cite{BJK} (Corollary \ref{representable object actions}).

\section{Preliminaries}

\subsection{Semi-abelian categories.\\}

Semi-abelian categories \cite{Janelidze2002367} are finitely complete, pointed, exact in the sense of M. Barr \cite{Barr}, protomodular in the sense of D. Bourn \cite{protomodularity}, with finite coproducts.
These categories have been introduced to capture some typical algebraic properties valid for non-abelian algebraic structures such as groups, Lie algebras, rings, crossed modules, varieties of $\Omega$-groups in the sense of P. Higgins \cite{Higgins01071956} and compact groups. As already mentioned in the introduction, every abelian category is in particular semi-abelian.

 Although protomodularity is a property that can be expressed in any category with finite limits, in the pointed context, i.e. when there is a zero object $0$ in $\mathbf{C}$, protomodularity amounts to the fact that the following formulation of the \emph{Split Short Five Lemma} holds:
 given a commutative diagram
 \[\xymatrixcolsep{5pc}\xymatrix{
0 \ar[r]^-{} &  K \ar[r]^-{k} \ar[d]_-{\kappa} & A \ar@<0.5ex>[r]^-{f} \ar[d]_-{\alpha} & B \ar[d]^-{\beta} \ar@<0.5ex>[l]^-{s} \\
0 \ar[r]^-{} &  K' \ar[r]_-{k'} & A' \ar@<0.5ex>[r]^-{f'} & B' \ar@<0.5ex>[l]^-{s'} }\]

\begin{flushleft}
where $k=ker(f)$, $k'=ker(f')$, $f\circ s = 1_{B}$, and $f'\circ s' = 1_{B'}$ (i.e. $f$ and $f'$ are split epimorphisms with sections $s$ and $s'$), if both $\kappa$ and $\beta$ are isomorphisms, then so is $\alpha$. 
\end{flushleft}

%A useful lemma which holds in protomodular categories is the following (\cite{protomodularity}, Proposition $11$):
%\begin{lemma}\label{jointlyepic}
%Given a split short exact sequence in a pointed protomodular category
% \[\xymatrixcolsep{5pc}\xymatrix{
%0 \ar[r]^-{} &  K \ar[r]^-{k}  & A \ar@<0.5ex>[r]^-{f} & B  \ar@<0.5ex>[l]^-{s} \ar[r] & 0 }\]
%\begin{flushleft}
%the pair of morphisms $(k,s)$ is jointly epimorphic.
%\end{flushleft}
%\end{lemma}

Any protomodular category $\mathbf{C}$ is a \textit{Mal'tsev} category \cite{CLP}: this means that every (internal) reflexive relation  $\mathbf{C}$ is an (internal) equivalence relation. Recall that a reflexive relation on an object $X$ is a diagram of the form
\begin{equation}\label{reflexive}
\xymatrix{
 R \ar@<1ex>[r]^-{p_{1}} \ar@<-1ex>[r]_-{p_{2}} &  X \ar[l]|-{\delta},  
  }
\end{equation}
 where $p_1$ and $p_2$ are jointly monic, and $p_{1}\circ \delta = 1_{X} = p_{2}\circ \delta$; such a reflexive relation $R$ is an equivalence relation  when, moreover, there exist $\sigma: R \longrightarrow R$ and $\tau : R \times_{X} R \longrightarrow R$ as in the following diagram
\[\xymatrixcolsep{6pc}
\xymatrix{
R \times_{X} R \ar[r]^-{\tau} &  R \ar@(ur,ul)_-{\sigma} \ar@<1ex>[r]^-{p_{1}} \ar@<-1ex>[r]_-{p_{2}} &  X \ar[l]|-{\delta}  
  }\]
  such that $p_{1}\circ \sigma = p_{2}$ and $p_{2}\circ \sigma = p_{1}$ (symmetry), and $p_{1} \circ \tau = p_{1}\circ \pi_{1}$ and $p_{2} \circ \tau = p_{2}\circ \pi_{2}$ (transitivity), where $\pi_{1}$ and $\pi_{2}$ are the projections in the following pullback:
    
\[\xymatrix{
R \times_{X} R \ar[r]^-{\pi_{2}} \ar[d]_-{\pi_{1}}   \ar@{}[rd]|<<<{\pullback}  & R \ar[d]^-{p_{1}}\\
R \ar[r]_-{p_{2}} & X
}\]

In the present article, by a \textit{regular} category is meant a finitely complete category where every morphism can be factorized as a regular epimorphism followed by a monomorphism, and where regular epimorphisms are pullback stable. A regular category $\mathbf C$ is said to be \textit{Barr-exact} if, moreover, every equivalence relation is effective, i.e. every equivalence relation is the kernel pair of a morphism in $\mathbf C$. A category which is pointed, protomodular and regular is said to be \textit{homological} \cite{BorceuxBourn}. In this context several basic diagram lemmas of homological algebra hold true (such as the snake lemma, the $3$-by-$3$-Lemma, etc.). 

We end these preliminaries with the following diagram indicating some implications between the different contexts recalled above:
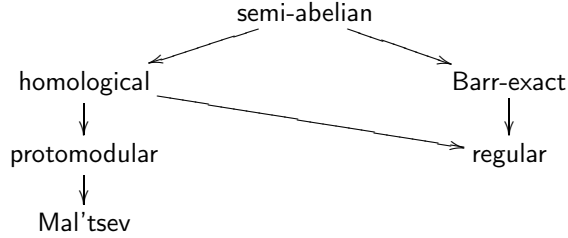
\[\xymatrixrowsep{1pc}\xymatrix{
  & \text{semi-abelian} \ar[ld] \ar[rd] & \\
\text{homological} \ar[d] \ar[rrd] & & \text{Barr-exact} \ar[d]\\
  \text{protomodular} \ar[d] & & \text{regular} \\
  \text{Mal'tsev} & & 
  }\] 

\subsection{The category $\mathbf{Hopf}_{K,coc}$ of cocommutative Hopf $K$-algebras.\\}

The category we study in this article is the category of Hopf $K$-algebras over a field $K$, denoted by $\mathbf{Hopf}_{K}$. The objects in $\mathbf{Hopf}_{K}$ are Hopf $K$-algebras, i.e. sextuples $(H,M,u,\Delta,\epsilon,S)$ where $(H,M,u)$ is a $K$-algebra and $(H,\Delta,\epsilon)$ is a $K$-coalgebra, such that these two structures are compatible, i.e. maps $M$ and $u$ are $K$-coalgebras morphisms, making $(H,M,u,\Delta,\epsilon)$ a $K$-bialgebra. The linear map $S$ is called the \textit{antipode}, and makes the following diagram commute:

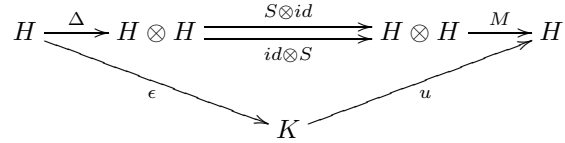
\[\xymatrix{
H \ar[rrd]_-{\epsilon} \ar[r]^-{\Delta} & H\otimes H \ar@<0.5ex>[rr]^-{S \otimes id} \ar@<-0.5ex>[rr]_-{id \otimes S} & &  H\otimes H \ar[r]^-{M}  & H  \\
& &K \ar[rru]_-{u}& & 
  }\]
\vspace*{0.2cm}
  
Morphisms in $\mathbf{Hopf}_{K}$ are exactly morphisms of $K$-bialgebras (i.e. morphisms that are both morphisms of $K$-algebras and $K$-coalgebras), as morphisms of $K$-bialgebras always preserve antipodes.

To denote the comultiplication map of a Hopf algebra $H$, we will use the Sweedler notation: $\forall h\in H$, $\Delta(h) = h_{1} \otimes h_{2}$ by omitting the summation sign. A Hopf algebra $H$ is said to be \textit{cocommutative} if its comultiplication map $\Delta$ makes the following diagram commute, where $\sigma$ is the linear map such that $\sigma(x\otimes y)= y\otimes x$, $\forall x,y \in H$

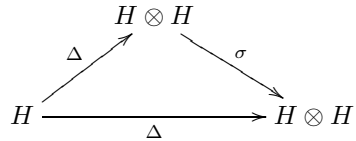
\[\xymatrix{
& H \otimes H \ar[rd]^-{\sigma}& \\
H\ar[rr]_-{\Delta} \ar[ru]^-{\Delta} & & H \otimes H
  }\] 

The category of cocommutative Hopf $K$-algebras will be denoted by $\mathbf{Hopf}_{K,coc}$. 
In the category $\mathbf{Hopf}_{K,coc}$ there are two full subcategories which will be of importance for our work: the category $\mathbf{GrpHopf}_{K}$ of group Hopf algebras, and the category $\mathbf{PrimHopf}_{K}$ of primitive Hopf algebras, whose definitions we are now going to recall.

\begin{enumerate}
\item

The \textit{group Hopf algebra on a group G}, denoted by $K[G]$, is the free vector space on $G$ over the field $K$, i.e.\ $K[G] = \{\sum_{g\in G} \alpha_{g}g$, where $(\alpha_{g})_{g\in G}$ is a family of scalars with only a finite number being non zero$\}$ and $\{g\mid g\in G\}$ is a basis of $K[G]$. The group Hopf algebra $K[G]$ can be equipped with a structure of cocommutative Hopf algebra, by taking the multiplication induced by the group law, and comultiplication $\Delta: K[G] \longrightarrow K[G] \otimes K[G]$, counit $\epsilon:K[G]\to K$ and antipode $S:K[G]\to K[G]$ the linear maps defined on the base elements respectively by $\Delta(g) = g\otimes g$,  $\epsilon(g)=1$ and $S(g)= g^{-1}$, $\forall g \in G$.

This assignment defines a functor $K[-] \colon \mathbf{Grp} \rightarrow \mathbf{Hopf}_{K}$ from the category of groups to the category of Hopf algebras, which has a right adjoint $\mathcal G \colon \mathbf{Hopf}_{K} \rightarrow \mathbf{Grp} $, that associates to any Hopf algebra $H$ its group of grouplike elements ${\mathcal G}(H)=\{x\in H~|~\Delta(x)=x\otimes x, \epsilon(x)=1\}$. We can now consider the category $\mathbf{GrpHopf}_{K}$, which is the full subcategory of $\mathbf{Hopf}_{K}$ whose objects are group Hopf algebras, that is Hopf algebras generated as algebra by grouplike elements. The functor $K[-] \colon \mathbf{Grp} \rightarrow \mathbf{GrpHopf}_{K}$ is an isomorphism of categories.
Indeed, it is clear that $K[-] \circ {\mathcal G}$ is the identity functor on the category of group Hopf algebras, and let us recall why the functor ${\mathcal G} \circ K[-]$ is the identity functor on the category of groups. If $G$ is a group and $K[G]$ the group Hopf K-algebra on $G$, one clearly has that $G \subseteq \mathcal{G} \big(K[G]\big)$ (in fact, this inclusion is the unit of the adjunction described above). Conversely, let $x= \sum_{g\in G} \alpha_{g}g\in K[G]$ be a group-like element. We have $\Delta(x) = \sum_{g\in G} \alpha_{g}g\otimes g$ but also $\Delta(x) = x\otimes x=\sum_{g,h\in G}\alpha_g\alpha_h g\otimes h$. Since $\{g\otimes h ~|~g,h\in G\}$ forms a basis of $K[G]\otimes K[G]$, we have $\forall g,h \in G$: $\alpha_g^2=\alpha_g$, and $\alpha_g\alpha_h=0$ if $g\neq h$. It follows that all $\alpha_g$'s should be zero except one that should be $1_K$, thus $x$ is in $G$.

\item The \textit{universal enveloping algebra of a Lie algebra L}, denoted by $U(L)$, is defined by the quotient $U(L)=T(L)/I$, where $T(L)$ is the tensor algebra on the vector space underlying $L$, and $I$ is the two-sided ideal of $T(L)$ generated by the elements of the form $x\otimes x' - x'\otimes x - [x,x']$, $\forall x,x' \in L$. Remark that the elements of $L$ generate $U(L)$ as an algebra. The universal enveloping algebra $U(L)$ can be equipped with a structure of cocommutative (and non commutative) Hopf algebra, by taking the concatenation as multiplication, and comultiplication $\Delta: U(L) \longrightarrow U(L) \otimes U(L)$, counit $\epsilon:U(L)\to K$ and antipode $S:U(L)\to U(L)$ the algebra maps defined on the generators by $\Delta(x) = x\otimes 1 + 1\otimes x$, $\epsilon(x)=0$ and $S(x)=-x$, 
$\forall x \in L$.
 
Recall that for any Hopf algebra $H$ an element $x\in H$ is called a \textit{primitive element} if $\Delta(x) = x\otimes 1 + 1\otimes x$ (and consequently, $\epsilon(x)=0$).
The above constructions lead to a pair of adjoint functors, where the functor $U:\mathbf{LieAlg}_{K}\to \mathbf{Hopf}_{K}$ is a left adjoint to $P:\mathbf{Hopf}_{K}\to \mathbf{LieAlg}_{K}$.
We can now consider the category $\mathbf{PrimHopf}_{K}$, which is the full subcategory of $\mathbf{Hopf}_{K}$
whose objects are primitive Hopf algebras, that is Hopf algebras generated as algebra by primitive elements. In the case where $K$ is of characteristic $0$, the category $\mathbf{PrimHopf}_{K}$ is known to be isomorphic to the category $\mathbf{LieAlg}_{K}$ of Lie $K$-algebras \cite[Theorem 5.18]{Milnor}.
\end{enumerate}
\vspace*{0.3cm}
\begin{remark}
\textsf{
As can be seen from the formula for comultiplication, both group Hopf algebras and primitive Hopf algebras are cocommutative. Therefore the categories $\mathbf{GrpHopf}_{K}$ and $\mathbf{PrimHopf}_{K}$ are also full subcategories of $\mathbf{Hopf}_{K,coc}$.
The functors $U, P, K[-], \mathcal{G}$ and their adjunctions are represented in the following diagram:
\[\xymatrixcolsep{6pc}
\xymatrix{
\mathbf{Grp} \ar@<1ex>[r]^-{K[-]}_-{\perp} & \mathbf{Hopf}_{K,coc} \ar@<1ex>[r]^-{P}_-{\top} \ar@<1ex>[l]^-{\mathcal G} & \mathbf{LieAlg}_{K}  \ar@<1ex>[l]^-{U} 
}\]
}
\end{remark}

\section{The category of cocommutative Hopf algebras over a field of characteristic zero is homological}
\label{homological}

The category $\mathbf{Hopf}_{K,coc}$ is certainly pointed, with the zero object $K$, that will be denoted by $0$, from now on. $\mathbf{Hopf}_{K,coc}$ is complete and cocomplete, since it is locally presentable \cite{Porst1}. We will now establish its protomodularity and regularity.

\subsection{Protomodularity of the category $\mathbf{Hopf}_{K}$}\label{protomodular}

Let us consider the following commutative diagram of short exact sequences in the category $\mathbf{Hopf}_{K,coc}$:

\begin{equation}\label{diagram}
\xymatrixcolsep{5pc} \xymatrix{
0 \ar[r]^-{} &  A \ar[r]^-{} \ar[d]_-{id_{A}} & C_{1} \ar[r]^-{} \ar[d]^-{\theta} & B \ar[d]^-{id_{B}} \ar[r]^-{} & 0 \\
0 \ar[r]^-{} &  A \ar[r]_-{} & C_{2} \ar[r]_-{} & B \ar[r]^-{} & 0 }
\end{equation}

Thanks to the explicit descriptions of equalizers and coequalizers given in \cite{Andru1} one can easily prove that the kernel and the cokernel of $\theta$ are the zero object. This proves that $\theta$ is a monomorphism and an epimorphism of Hopf $K$-algebras. Since monomorphisms are injections and epimorphisms are surjections in the category $\mathbf{Hopf}_{K,coc}$ \cite{Chirva1,Nasta1}, this shows that $\theta$ is an isomorphism of Hopf $K$-algebras and so the category $\mathbf{Hopf}_{K,coc}$ is protomodular.

The category $\mathbf{Hopf}_{K,coc}$ is actually even \emph{strongly protomodular} \cite{strong}, since it has finite limits and it can be viewed as the category of internal groups in the category of cocommutative $K$-coalgebras (see \cite{Street}, for instance). 

%In fact, returning to the case of a base field $K$, the protomodularity holds more generally for the category $\mathbf{Hopf}_{K}$ of arbitrary Hopf $K$-algebras. This follows from the following result (Lemma 3.2.19 in \cite{Andru1}) by taking into account the fact that any split extension of Hopf algebras induces a \emph{cleft exact sequence} in the sense of by Andruskiewitsch and Devoto \cite{Andru1}:
%
%
%
%\begin{theorem}
%Consider a diagram of the form \eqref{diagram} in the category of Hopf $K$-algebras, where the above exact sequence is cleft. Then the bottom exact sequence is also cleft and $\theta$ is an isomorphism.
%\end{theorem}

\subsection{Semi-direct products of cocommutative Hopf algebras}

Let $B$ be a cocommutative Hopf algebra. A {\em $B$-module Hopf algebra} is a Hopf algebra $A$ that is at the same time a left $B$-module with action $\rho :B\ot A\to A, \rho(b\ot a)=b\cdot a$ such that $\rho$ is a morphism of bialgebras. The semi-direct product (also known as {\em smash product}) of $B$ and $A$, denoted by $A \rtimes B$, is the Hopf algebra whose underlying vector space is the tensor product $A\ot B$ and with the following structure: the unit is $u_{A\rtimes B}=u_A\ot u_B$ and multiplication given by
$$(a\ot b)(a'\ot b')=a(b_1\cdot a')\ot b_2b',$$
for all $a,a'\in A$ and $b,b'\in B$. The coalgebra structure is given by the tensor product coalgebra, i.e. $\Delta_{A\rtimes B}= (id_A \otimes \sigma \otimes id_{B})(\Delta_{A} \otimes \Delta_{B})$ and $\epsilon_{A\rtimes B}=\epsilon_A\ot \epsilon_B$. The antipode is given by 
$S_{A\rtimes B}(a\ot b) = S_B(b_1)\cdot S_A(a)\ot S_B(b_2)$.\\

The following Lemma is a reformulation of Theorem $4.1$ in \cite{Molnar}:
\begin{lemma}\label{Molnar}
Every split short exact sequence in $\mathbf{Hopf}_{K,coc}$ 
 \[\xymatrixcolsep{5pc}\xymatrix{
0 \ar[r]^-{} &  A \ar[r]^-{k}  & H \ar@<0.5ex>[r]^-{p} & B  \ar@<0.5ex>[l]^-{s} \ar[r] & 0 }\]
is canonically isomorphic to the semi-direct product exact sequence
\[\xymatrixcolsep{5pc}\xymatrix{
0 \ar[r]^-{} &  A \ar[r]^-{i_1}  & A \rtimes B \ar@<0.5ex>[r]^-{p_2} & B  \ar@<0.5ex>[l]^-{i_2} \ar[r] & 0 }\]
where $i_1=id_A\ot u_B$, $i_2=u_A\ot id_B$ and $p_2=\epsilon_A\ot id_B$.
\end{lemma}
\begin{proof}
The arrow $ h \colon A \rtimes B \rightarrow H$ in the diagram below is given by $h(a\otimes b)= k(a) s(b)$ for all $a\otimes b \in A \rtimes B$.
 \[\xymatrixcolsep{5pc}\xymatrix{
0 \ar[r]^-{} &  A \ar[r]^-{i_1}  & A \rtimes B \ar@<0.5ex>[r]^-{p_2} \ar[d]^h & B  \ar@<0.5ex>[l]^-{i_2} \ar[r] & 0 \\
0 \ar[r]^-{} &  A \ar@{=}[u] \ar[r]^-{k}  & H \ar@<0.5ex>[r]^-{p} & B  \ar@<0.5ex>[l]^-{s} \ar[r]\ar@{=}[u] & 0 }
\]
This is a morphism of split short exact sequences, and therefore $h$ is an isomorphism by protomodularity of $\mathbf{Hopf}_{K,coc}$.
\end{proof}

We use this lemma to reformulate the well-known structure theorem for cocommutative Hopf algebras over an algebraically closed field of characteristic zero (see for instance \cite{Sweedler}, page 279 in combination with Lemma 8.0.1(c)) in terms of split exact sequences.

\begin{theorem}[Cartier-Gabriel-Moore-Milnor-Kostant]\label{MM} Every cocommutative Hopf K-algebra $H$, over an algebraically closed field $K$ of characteristic $0$, is isomorphic to the semi-direct product $$H \cong U(L_{H})\rtimes K[G_{H}]$$ of the universal enveloping algebra of a Lie algebra $U(L_{H})$ with the group Hopf algebra $K[G_{H}]$, where $L_{H}$ and $G_{H}$ are given respectively by the space of primitive elements and the group of group-like elements of $H$. Consequently, with notations of Lemma \ref{Molnar}, for each $H$ there exists a canonical split exact sequence of cocommutative Hopf algebras of the following form
 \[\xymatrix{
0 \ar[r]^-{} & U(L_{H})  \ar[r]^-{i_{H}} & H \ar@<0.5ex>[r]^-{p_{H}} & K[G_{H}]  \ar@<0.5ex>[l]^-{s_{H}} \ar[r]^-{} & 0 }\]
where $i_{H} = h \circ i_{1}$, $s_{H} = h \circ i_{2}$ and $p_{H} = p_{2} \circ h^{-1}$.
\end{theorem}

\begin{remark}\label{morphism}
Every morphism $f: H_{1} \longrightarrow H_{2}$ of cocommutative Hopf $K$-algebras gives rise to a morphism of split exact sequences of following form

\[\xymatrixcolsep{5pc}\xymatrix{
0 \ar[r]^-{} &  U(L_{H_{1}}) \ar[r]_-{i_{H_1}} \ar[d]_-{f_{1}:=U(P(f))} & H_1 \ar[r]_-{p_{H_1}} \ar[d]_-{f}  & K[G_{H_{1}}] \ar[d]^-{f_{2}:=K[\mathcal G(f)]} \ar[r]^-{} \ar@/_/[l]_-{s_{H_1}}& 0 \\
0 \ar[r]^-{} &  U(L_{H_{2}}) \ar[r]_-{i_{H_2}} & H_2 \ar[r]_-{p_{H_2}}   & K[G_{H_{2}}] \ar[r]^-{} \ar@/_/[l]_-{s_{H_2}} & 0 }\]
\end{remark}

\subsection{Regularity of the category $\mathbf{Hopf}_{K,coc}$}

\subsubsection{The regular epimorphism/monomorphism factorization in $\mathbf{Hopf}_{K,coc}$}

Let $f: A \longrightarrow B$ be a morphism of cocommutative Hopf $K$-algebras. By the protomodularity of $\mathbf{Hopf}_{K,coc}$, it is well known that regular epimorphisms are the same as cokernels, i.e. normal epimorphisms. Thus, to construct the regular epimorphism/monomorphism factorization of the morphism $f$, we  consider the kernel $i:Hker(f)\to A$ of $f$ and the cokernel $p:A\to HCoker (i)$ of $i$, both computed in the category of $\mathbf{Hopf}_{K,coc}$ :
%
% the cokernel $p$ of the kernel of $f$, then we have the following commutative diagram:

 \[\xymatrixcolsep{8pc}\xymatrixrowsep{1pc}\xymatrix{
   Hker(f) \ar@{^{(}->}[r]^-{i}& A \ar[r]^-{f} \ar@{->>}[d]_-{p} & B \\
   & HCoker (i) \ar@{>->}[ur]_-{m} &  
  }\]

The existence of this factorization $m$ such that $m\circ p=f$ follows from the universal property of the cokernel $p$ of $i$. It remains to prove that $m$ is a monomorphism, which is equivalent in $\mathbf{Hopf}_{K,coc}$ to showing that $m$ is an injection. In the category $\mathbf{Hopf}_{K,coc}$ the above factorization is obtained as in the category of vector spaces since
$HCoker(i)=\dfrac{A}{A\big(Hker(f)\big)^{+} A}$
(by \cite{Andru1}),  and $ker(f) = A\big(Hker(f)\big)^{+} A,$ (by \cite{shudo1988,Newman19751}).

Note that any epimorphism of cocommutative Hopf $K$-algebras is then a normal epimorphism, and the following classes of epimorphisms coincide in $\mathbf{Hopf}_{K,coc}$:

\hspace*{2cm} normal epis = regular epis = epis = surjective morphisms.

\subsubsection{Pullback stability of regular epimorphisms in the category $\mathbf{Hopf}_{K,coc}$}

To prove the pullback stability of regular epimorphisms in the category $\mathbf{Hopf}_{K,coc}$, the approach we follow is to apply the pullback stability of regular epimorphisms in the two full subcategories $\mathbf{GrpHopf}_{K}$ and $\mathbf{PrimHopf}_{K}$ of $\mathbf{Hopf}_{K,coc}$, which are both semi-abelian, and closed under pullbacks and quotients in  $\mathbf{Hopf}_{K,coc}$. From the regularity of these two categories and the decomposition Theorem \ref{MM} we shall deduce the regularity of $\mathbf{Hopf}_{K,coc}$.

\begin{lemma}\label{Subcategories}
The two full subcategories $\mathbf{GrpHopf}_{K}$ and $\mathbf{PrimHopf}_{K}$, of  $\mathbf{Hopf}_{K,coc}$, are semi-abelian categories. Both these categories are closed under quotients and pullbacks in $\mathbf{Hopf}_{K,coc}$.
\end{lemma}

\begin{proof}
As recalled in Preliminaries 2.2, the two full subcategories $\mathbf{GrpHopf}_{K}$ and $\mathbf{PrimHopf}_{K}$ are isomorphic to the category $\mathbf{Grp}$ of groups and to the category $\mathbf{LieAlg}_{K}$ of Lie $K$-algebras, respectively. The fact that $\mathbf{Grp}$ and $\mathbf{LieAlg}_{K}$ are semi-abelian is well-known (see \cite{BorceuxBourn}, for instance).

The categories $\mathbf{GrpHopf}_{K}$ and $\mathbf{PrimHopf}_{K}$ are closed under quotients since morphisms of Hopf $K$-algebras preserve both group-like and primitive elements. These categories are closed under products in the category $\mathbf{Hopf}_{K,coc}$, as explained in \cite{Milnor}. To see that $\mathbf{PrimHopf}_{K}$ is closed under subobjects in $\mathbf{Hopf}_{K,coc}$, let us consider a monomorphism $m: A \rightarrowtail U(L)$ in $\mathbf{Hopf}_{K,coc}$ with codomain a primitive Hopf algebra. The morphism $m$, being a Hopf algebra morphism, preserves group-like elements and the group of group-like elements is trivial in the primitive Hopf algebra $U(L)$. Since $m$ is injective, we see that the group of group-like elements of $A$ is trivial as well. By applying Theorem \ref{MM}, we conclude that $A$ has to be a primitive Hopf algebra as well. Similar arguments show that $\mathbf{GrpHopf}_{K}$ is closed under subobjects in $\mathbf{Hopf}_{K,coc}$. It follows that the subcategories $\mathbf{GrpHopf}_{K}$ and $\mathbf{PrimHopf}_{K}$ are closed under pullbacks in $\mathbf{Hopf}_{K,coc}$.
\end{proof}

\begin{remark}

 In the following we shall assume that $K$ is an algebraically closed field. It can be checked that this is not  a restriction: indeed, given a field $K$ and $\phi:K\to\overline K$ an embedding of $K$ in an algebraic closure $\overline{K}$, one has the adjunction 
\[\xymatrixcolsep{6pc}
\xymatrix{
\mathbf{Hopf}_{\overline{K},coc}  \ar@/_/[r]_-{R_\phi}^-{\perp} &\mathbf{Hopf}_{K,coc}  \ar@/_/[l]_-{L_\phi}
}\]
where $R_\phi$ is the ``restriction of scalars functor'' and $L_\phi=- \otimes_K \overline K$ its left adjoint, the ``extension of scalars'' functor. Being a left adjoint, $L_\phi$ preserves regular epimorphisms and moreover $L_\phi$ reflects regular epimorphisms and preserves finite limits. 
Accordingly, knowing that $\mathbf{Hopf}_{\overline{K},coc}$ is regular (respectively, exact), one can deduce from this that $\mathbf{Hopf}_{{K},coc}$ is regular (resp. exact) as well.

\end{remark}

The following result concerning split short exact sequences in $\mathbf{Hopf}_{K,coc}$ will be useful in the proof of the regularity of this category:

\begin{lemma}\label{surjective}

Given the following commutative diagram of split short exact sequences in $\mathbf{Hopf}_{K,coc}$:
\[\xymatrixcolsep{5pc}\xymatrix{
0 \ar[r]^-{} &  A_1 \ar[r]_-{i_{H_1}} \ar[d]_-{{h}_{A}} & H_1 \ar[r]_-{p_{H_1}} \ar[d]_-{h}  & B_1 \ar[d]^-{{h}_B} \ar[r]^-{} \ar@/_/[l]_-{s_{H_1}}& 0 \\
0 \ar[r]^-{} &  A_2 \ar[r]_-{i_{H_2}} & H_2 \ar[r]_-{p_{H_2}}   & B_2 \ar[r]^-{} \ar@/_/[l]_-{s_{H_2}} & 0 }\]

We have that $h$ is surjective if and only if both $h_A$ and $h_B$ are surjective.
\end{lemma}

\begin{proof}
We apply Lemma \ref{Molnar} to the exact sequences in the statement of the Lemma, we obtain the following commutative diagram which is canonically isomorphic to the previous one:
 \[\xymatrixcolsep{5pc}\xymatrix{
0 \ar[r]^-{} &  A_1 \ar[r]_-{i_{1}} \ar[d]_-{{h}_{A}} & A_1 \rtimes B_1  \ar[r]_-{p_{1}} \ar[d]_-{h_A \otimes h_B} \ar@/_/@{.>}[l]_-{\xi_{1}} & B_1 \ar[d]^-{{h}_B} \ar[r]^-{} \ar@/_/[l]_-{s_{1}}& 0 \\
0 \ar[r]^-{} &  A_2 \ar[r]_-{i_{2}} & A_2 \rtimes B_2  \ar[r]_-{p_{2}}  \ar@/_/@{.>}[l]_-{\xi_{2}} & B_2 \ar[r]^-{} \ar@/_/[l]_-{s_{2}} & 0 }\]
Hence, the morphism $h$ is surjective if and only if $h_A \otimes h_B \colon A_1 \rtimes B_1 \to A_2 \rtimes B_2$ is surjective.

If $h_A$ and $h_B$ are surjective, then $h_A \otimes h_B$ is surjective by considering this morphism on its underlying vector space. For the converse implication, if $ h_A \otimes h_B $ is surjective, let us note that 
 for each semi-direct product $A_i\rtimes B_i$, the underlying coalgebra is exactly the categorical product of the coalgebras $A_i$ and $B_i$; we denote $\xi_i=id_{A_i}\ot \epsilon_{B_i}$ for the coalgebra-projection of $A_i\rtimes B_i$ onto $A_i$ (which is not a Hopf algebra morphism).
 
It is clear that $h_A \circ \xi_1 = \xi_2 \circ  (h_A \otimes h_B)$, as coalgebra morphisms. Since $\xi_{2}$ is a split epimorphism and $h_A \otimes h_B$ is surjective, we conclude that $h_A$ is surjective. It is clear that $h_{B}$ is surjective whenever $h$ is.

\end{proof}

\begin{theorem}
Consider the following pullback $(P,\pi_{A},\pi_{B})$ in the category $\mathbf{Hopf}_{K,coc}$:
\begin{equation}\label{pullback}
\xymatrix{
P \ar[r]^-{\pi_{B}} \ar[d]_-{\pi_{A}}   \ar@{}[rd]|<<<{\pullback}  & B \ar[d]^-{g}\\
A \ar[r]_-{f} & C
}
\end{equation}

\vspace*{0.2cm}
if $f$ is a regular  epimorphism then $\pi_{B}$ is a regular epimorphism.
%\end{flushleft}
\end{theorem} 
\begin{proof}
Thanks to Lemma \ref{Subcategories}, regular epimorphisms are pullback stable whenever the Hopf algebras $A$, $B$ and $C$ in diagram \eqref{pullback} belong to $\mathbf{GrpHopf}_{K}$, or to $\mathbf{PrimHopf}_{K}$, respectively.

Let us now consider $A$, $B$ and $C$ cocommutative Hopf $K$-algebras over a field $K$ of characteristic zero. By Theorem \ref{MM}, we have: $A \cong U(L_{A}) \rtimes K[G_{A}]$, $B \cong U(L_{B}) \rtimes K[G_{B}]$ and $C \cong U(L_{C}) \rtimes K[G_{C}]$.

Thanks to Remark \ref{morphism}, we can build the following commutative diagram.

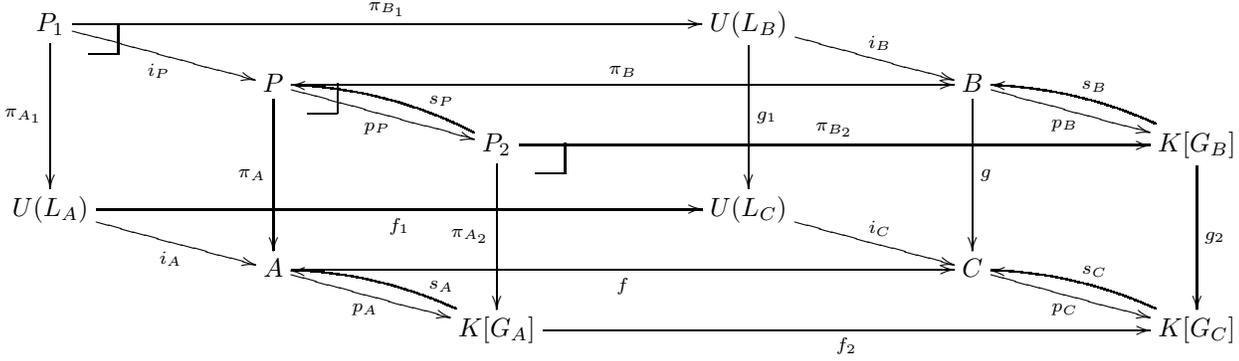
\[\xymatrixcolsep{5pc}\xymatrixrowsep{0.7pc}
\hspace{-1cm}  \xymatrix{
    P_{1} \ar[rrr]^-{{\pi}_{B_1}} \ar[ddd]_-{\pi_{A_1}} \ar[rd]_-{i_{P}} \ar@{}[rd]|<<<<<{\pullback} &&& U(L_B) \ar[rd]^-{i_{B}} \ar[ddd]^-{g_{1}}  \\
    & P \ar[rrr]^-{\pi_{B}}   \ar[ddd]_-{\pi_{A}} \ar[rd]_-{p_{P}} \ar@{}[rd]|<<<<<{\pullback} &&& B \ar[rd]_-{p_{B}} \ar[ddd]^-{g}  \\
    && P_{2} \ar[rrr]^-{{\pi}_{B_2}} \ar@/_/[lu]_(.3){s_{P}} \ar[ddd]_-{\pi_{A_2}} \ar@{}[rd]|<<<<<{\pullback} &&& K[G_B] \ar@/_/[lu]_{s_{B}} \ar[ddd]^-{g_{2}} \\
    U(L_A) \ar[rrr]_-{f_{1}}  \ar[rd]_-{i_{A}} &&& U(L_C) \ar[rd]^-{i_{C}} \\
    & A \ar[rrr]_-{f}  \ar[rd]_-{p_{A}} &&& C \ar[rd]_-{p_{C}}  \\
    && K[G_A] \ar@/_/[lu]_(.3){s_{A}} \ar[rrr]_-{f_{2}} &&& K[G_C] \ar@/_/[lu]_{s_{C}} \\
  }\] 
where $(P_{1},\pi_{A_1},\pi_{B_1})$ is the pullback of $f_{1}$ and $g_{1}$, $(P_{2},\pi_{A_2},\pi_{B_2})$ is the pullback of $f_{2}$ and $g_{2}$.

When $f$ is surjective, the surjectivity of $f_{1}$ and $f_{2}$ follows both from Lemma \ref{surjective} applied to the lower part of the diagram. The front and back squares of the diagram are in $\mathbf{GrpHopf}_{K}$ and in $\mathbf{PrimHopf}_{K}$, respectively, thus both ${\pi}_{B_1}$ and ${\pi}_{B_2}$ are surjective (by Lemma \ref{Subcategories}). Applying Lemma \ref{surjective} again (in converse direction), we obtain that $\pi_B$ is also surjective.

\end{proof}

\section{A torsion theory in the category $\mathbf{Hopf}_{K,coc}$}\label{torsion}
In the non-abelian context of homological categories it is natural to define and study a general notion of torsion theory, that extends the one introduced by S.E. Dickson in the frame of abelian categories \cite{Dickson}. This study was first initiated in \cite{Bourn200618}, and further developed in \cite{Duckerts20121837} \cite{EveraertGran}, also in relationship with semi-abelian homology theory.

Let us recall the definition of a torsion theory in the homological context:

\begin{definition}
\textsf{
In a homological category $\mathbf{C}$, a \textit{torsion theory} is given by a pair ($\mathbf{T}$,$\mathbf{F}$) of full and replete (i.e. isomorphism closed) subcategories of $\mathbf{C}$ such that:
\begin{enumerate}[i.]
\item For any object $X$ in $\mathbf{C}$, there exists a short exact sequence:
 \[\xymatrix{
0 \ar[r]^-{} &  T \ar[r]^-{t_{X}} & X \ar[r]^-{\eta_{X}} & F \ar[r]^-{} & 0 }\]
where $0$ is the zero object in $\mathbf{C}$, $T \in \mathbf{T}$ and $F \in \mathbf{F}$.
\item The only morphism $f: T\longrightarrow F$ from $T \in \mathbf{T}$ to $F \in \mathbf{F}$ is the zero morphism.
\end{enumerate}
}
\end{definition}
When ($\mathbf{T}$,$\mathbf{F}$) is a torsion theory, $\mathbf{T}$ is called the \textit{torsion} subcategory of $\mathbf{C}$, and $\mathbf{F}$ its \textit{torsion-free} subcategory.
Among the many examples in the homological context, let us just mention the following ones:

\begin{example}
\textsf{
\begin{enumerate}
\item
 Every torsion theory in an abelian category $\mathbf{C}$. For instance: the pair ($\mathbf{Ab}_{t}$, $\mathbf{Ab}_{tf}$) in the category $\mathbf{Ab}$ of abelian groups, where $\mathbf{Ab}_{t}$ and $\mathbf{Ab}_{tf}$ denote the full and replete subcategories of the category of abelian groups whose objects are torsion and torsion-free abelian groups, respectively.
\item
 The pair ($\mathbf{NilCRng}$, $\mathbf{RedCRng}$) in the category $\mathbf{CRng}$ of commutative rings, where $\mathbf{NilCRng}$ and $\mathbf{RedCRng}$ denote the full subcategories of nilpotent commutative rings, and of reduced commutative rings (i.e. commutative rings with no nontrivial nilpotent elements), respectively.
\item
 The pair ($\mathbf{Grp(Ind)}$, $\mathbf{Grp(Haus)}$) in the category $\mathbf{Grp(Top)}$ of topological groups, where $\mathbf{Grp(Ind)}$ and $\mathbf{Grp(Haus)}$ denote the full subcategories of indiscrete groups and of Hausdorff groups, respectively.
\end{enumerate}
}
\end{example}

From now on, in the homological category of cocommutative Hopf $K$-algebras, $\mathbf{T}$ will always denote the (full) subcategory $\mathbf{PrimHopf}_{K}$ of primitive Hopf algebras, and $\mathbf{F}$ the (full) subcategory $\mathbf{GrpHopf}_{K}$ of group Hopf algebras.

\begin{theorem}\label{hereditary torsion}
The pair ($\mathbf{PrimHopf}_{K}$,$\mathbf{GrpHopf}_{K}$) is a hereditary torsion theory in $\mathbf{Hopf}_{K,coc}$.
\end{theorem}

\begin{proof}
Thanks to Theorem \ref{MM}, we know that we can associate the following short exact sequence with any cocommutative Hopf $K$-algebra $H$:
 \[\xymatrix{
0 \ar[r]^-{} &  U(L_{H}) \ar[r]^-{i_{H}} & H \ar[r]^-{p_{H}} & K[G_{H}] \ar[r]^-{} & 0 }\]

Any morphism $f \colon U(L) \rightarrow K[G]$ from a primitive Hopf algebra $U(L) \in \mathbf{PrimHopf}_{K}$ to a group Hopf algebra $K[G] \in \mathbf{GrpHopf}_{K}$ is the zero morphism in $\mathbf{Hopf}_{K,coc}$. Indeed, any primitive Hopf algebra $U(L)$ is generated by its primitive elements, which are preserved by any morphism of Hopf algebras. Since a group Hopf algebra $K[G]$ does not contain any non-zero primitive element, it follows that $f$ is the zero morphism. It follows that ($\mathbf{PrimHopf}_{K}$,$\mathbf{GrpHopf}_{K}$) is a torsion theory, which is actually hereditary thanks to Lemma \ref{Subcategories}.

\end{proof}

As it follows from the results in \cite{Bourn200618} the reflector $I$ in the adjunction
\[\xymatrixcolsep{6pc}
\xymatrix{
\mathbf{F} \ar@/_/[r]_-{H}^-{\perp} &\mathbf{Hopf}_{K,coc}  \ar@/_/[l]_-{I}
}\]
is semi-left-exact in the sense of Cassidy-Hebert-Kelly \cite{JAZ:4910656}, i.e. it preserves all pullbacks of the form

\begin{equation*}
\xymatrix{
P \ar[r]^-{p_{2}} \ar[d]_-{p_{1}}   \ar@{}[rd]|<<<{\pullback}  & Y \ar[d]^-{\eta_{Y}}\\
H(X) \ar[r]_-{H(f)} & HI(Y)
}
\end{equation*}
where $\eta_{Y}$ is the $Y$-component of the unit of the adjunction and $f$ lies in the subcategory $\mathbf{F}$.
The adjunction is then \emph{admissible} in the sense of categorical Galois theory \cite{Janelidze}: this opens the way to further investigations in the direction of semi-abelian homology \cite{Duckerts20121837}. The fact that the torsion theory is hereditary and $\mathbf{Hopf}_{K,coc} $ a homological category implies that  the corresponding \emph{Galois coverings} are precisely those regular epimorphisms $f \colon A \rightarrow B$ in $\mathbf{Hopf}_{K,coc} $ with the property that the kernel $Hker(f)$ is in $\mathbf{F}$ (by applying Theorem $4.5$ in \cite{GranRossi}). This fact is crucial to describe generalized Hopf formulae for homology, as explained in \cite{EveraertGran}.

\section{The category of cocommutative Hopf algebras over a field of characteristic zero is semi-abelian}
\label{exact}

In order to prove that $\mathbf{Hopf}_{K,coc}$ is semi-abelian, it remains to show that equivalence relations are effective.
For this, we shall show that any equivalence relation $R$  as in diagram \eqref{reflexive}  in $\mathbf{Hopf}_{K,coc}$ is the kernel pair of its coequalizer $q \colon X \rightarrow \overline{X}$. We first apply Theorem \ref{MM} to the equivalence relation $R$, obtaining the following commutative diagram, where the morphisms $q_{1}$, $q_{2}$ and $q$ are the coequalizers of $p_{11}$ and $p_{21}$, $p_{12}$ and $p_{22}$, $p_{1}$ and $p_{2}$, respectively, and ($Eq(q)$,$\pi_{1}$,$\pi_{2}$) is the kernel pair of $q$.

\[\xymatrixcolsep{4.5pc} \xymatrix{ 
  0 \ar[r]^-{} & U(L_{R}) \ar@/_/[rrd]_<<<{i'_{R}} \ar[r]^-{i_{R}} \ar@<2ex>[dd]^-{p_{21}} \ar@<-2ex>[dd]_-{p_{11}} & R  \ar@<2ex>[dd]^-{p_{2}} \ar@<-2ex>[dd]_-{p_{1}} \ar[rr]_-{p_{R}} \ar[rd]^-{\theta}& & K[G_{R}] \ar@<2ex>[dd]^-{p_{22}} \ar@<-2ex>[dd]_-{p_{12}}  \ar[r]^-{} & 0 \\
 & & & Eq(q) \ar@<0.3ex>[ld]^-{\pi_{2}} \ar@<-0.3ex>[ld]_-{\pi_{1}} \ar@/_/[ru]^-{p'_{R}} & & \\
 0 \ar[r]^-{} & U(L_{X}) \ar[r]^-{i_{X}} \ar[d]_-{q_{1}} \ar[uu]|-{\Delta_{X_{1}}} & X  \ar[rr]_-{p_{X}} \ar[uu]|-{\Delta_{X}} \ar[d]_-{q}& & K[G_{X}] \ar[d]^-{q_{2}} \ar[uu]|-{\Delta_{X_{2}}}  \ar[r]^-{} & 0  \\
 &  \overline{U(L_{X})} \ar[r]_-{\overline{i_{X}}} & \overline{X} \ar[rr]_-{\overline{p_{X}}}& &\overline{K[G_{X}]}   & 
 }\]

 Thanks to Lemma \ref{Subcategories}, the left column is exact, i.e. $U(L_{R}) = {Eq}(q_1)$. On the other hand, let us explain why the right column is also exact. First observe that $(K[G_{R}],p_{12}, p_{22})$ is a reflexive relation on $K[G_{X}]$, while the category $\mathbf{GrpHopf}_{K}$ of group Hopf algebras is exact Mal'tsev and closed under pullbacks and quotients in $\mathbf{Hopf}_{K,coc}$. It follows that $K[G_{R}]$ is the kernel pair of $q_2$ in $\mathbf{Hopf}_{K,coc}$.

The universal property of ($Eq(q)$,$\pi_{1}$,$\pi_{2}$) gives the unique arrow $\theta \colon R \rightarrow Eq(q)$ with $\pi_1 \circ \theta = p_1$ and $\pi_2 \circ \theta = p_2$. One can check that the arrow $\overline{i_{X}}$ is a monomorphism, by using the fact that the coequalizers of equivalence relations in $\mathbf{Hopf}_{K,coc}$ are computed as in $\mathbf{Coalg}_{K,coc}$. Consequently, the lower row in the diagram above is an exact sequence as is then the lower row in the following diagram.

\[\xymatrixcolsep{5pc}\xymatrix{
0 \ar[r]^-{} & U(L_{R})  \ar[r]^-{i_{R}} \ar[d]_-{id_{U(L_{R})}} & R \ar[r]^-{p_{R}} \ar[d]^-{\theta} & K[G_{R}] \ar[d]^-{id_{K[G_{R}]}} \ar[r]^-{} & 0 \\
0 \ar[r]^-{} & U(L_{R}) \ar[r]_-{i'_{R}} & Eq(q) \ar[r]_-{p'_{R}} & K[G_{R}] \ar[r]^-{} & 0 }\]
By applying the Split Short Five Lemma to the above commutative diagram, it follows that the morphism $\theta$ is an isomorphism, and the equivalence relation $R$ is effective.
One accordingly has the following:
\begin{theorem}\label{semi-abelian}
For any field $K$ of characteristic $0$, the category $\mathbf{Hopf}_{K,coc}$ of cocommutative Hopf $K$-algebras is semi-abelian. 
\end{theorem}
This result has the following interesting   
\begin{corollary}\label{representable object actions}
For any field $K$ of characteristic $0$, the category $\mathbf{Hopf}_{K,coc}$ is action representable \\ (i.e. it has representable object actions in the sense of \cite{BJK}).
\end{corollary}
\begin{proof}
This follows from Theorem 4.4 in \cite{BJK}, the exactness of $\mathbf{Hopf}_{K,coc}$ (Theorem \ref{semi-abelian}), and the fact that $\mathbf{Hopf}_{K,coc}$ can be viewed as the category of internal groups in the cartesian closed category of cocommutative $K$-coalgebras (see 2.3 in \cite{Grunenfelder}).
\end{proof}

\section*{Acknowledgement}
The authors  are very grateful to Clemens Berger and Stephen Lack for many useful suggestions on the subject of the article.

\section*{}
\begin{flushleft}
\begin{small}
\textbf{Funding} : The second author is funded by a grant from the National Fund for Scientific Research of Belgium.
\end{small}
\end{flushleft}

\bibliography{biblio}
\bibliographystyle{spmpsci}     

\end{document}